\documentclass[12pt,letterpaper]{amsart}
\usepackage{amsmath,amsfonts,amssymb, tikz, amsthm}
\usepackage[colorlinks=true,linkcolor=teal,anchorcolor=teal,citecolor=teal,filecolor=black,menucolor=black,runcolor=black,urlcolor=blue]{hyperref}
\usepackage[utf8]{inputenc}
\usepackage[english]{babel}
\usepackage{xcolor}
\usepackage[tableposition=top]{caption} 

\newtheorem*{theorem*}{Theorem}
\newtheorem{theorem}{Theorem}
\newtheorem{lemma}{Lemma}

\newtheorem{corollary}{Corollary}
\newtheorem{conjecture}{Conjecture}

\newtheorem{definition}{Definition}

\theoremstyle{remark}

\newtheorem{remark}{Remark}

\numberwithin{equation}{section}
\numberwithin{corollary}{theorem}
\numberwithin{remark}{section}

\newcommand{\Z}{\mathbb{Z}}

\newcommand{\C}{\mathbb{C}}
\newcommand{\R}{\mathbb{R}}

\newcommand{\X}{\mathfrak{X}}
\newcommand{\h}{\mathcal{H}}

\newcommand{\sym}{\mbox{Sym}}

\begin{document}
\title{A Note on Holomorphic Quantum Unique Ergodicity}

\author{Krishnarjun Krishnamoorthy}
\email[Krishnarjun K]{krishnarjunk@hri.res.in, krishnarjunmaths@gmail.com}
\address{Harish-Chandra Research 
	Institute, HBNI,
	Chhatnag Road, Jhunsi, Prayagraj 211019, India.}

\maketitle

\begin{abstract}
	In this paper we give a new proof of the Quantum Unique Ergodicity conjecture for holomorphic integral weight modular forms on the upper half plane. The proof requires only partial results towards the Ramanujan conjecture and the shifted convolution problem. Furthermore the proof is applicable to a wider class of cusp forms other than Hecke eigenforms. We also prove some corollaries, particularly towards the Lehmer's conjecture on the non vanishing of the Fourier coefficients. 
\end{abstract}

\section{Introduction}\label{Section "Introduction"}

Suppose $ \h :=\{z=x+iy\in \C\ |\ y >0\} $ denotes the upper half plane. The distribution of $ L^2 $ mass of certain smooth functions on the hyperbolic plane $ \mathfrak{X}:=\h/SL(2,\Z) $ is a very interesting problem both from the perspective of number theory and from the perspective of quantum chaos. In this direction, the following conjecture of Rudnick and Sarnak \cite{Rudnick-Sarnak QUE}, known as the quantum unique ergodicity (QUE) conjecture asserts that the $ L^2 $ masses of cuspidal eigen functions of the hyperbolic laplacian $ \Delta $ equidistribute as the ``energy" (captured by the size of eigenvalue) goes to infinity. This conjecture (now a theorem) has interpretations in connection to quantum mechanics.

\begin{conjecture}[QUE for $ \X $]\label{Conjecutre "QUE General"}
	Let $ \phi_i $ be a sequence of $ L^2 $ normalized eigenfunctions of the Laplacian $ \Delta $ on $ \X $ with associated eigenvalue $ \lambda_i\to\infty $. Then the probability measure $ d_iz = |\phi_i|^2 dz $ converges in the weak - * sense to the normalized measure $ (3/\pi) dz $.
\end{conjecture}

Conjecture \ref{Conjecutre "QUE General"} is expected to hold in much more generality, in particular for (strictly) negatively curved compact Riemann surfaces. The QUE conjecture for Maass forms on $ \mathfrak{X} $ was settled by Soundararajan \cite{Sound Escape of Mass} who built upon earlier works of Lindenstrauss \cite{Lindenstrauss} (see also \cite{Einsiedler notes}, \cite[Chapter 9]{Bergeron - Book} for more details). The analogue of QUE conjecture for holomorphic modular forms was resolved by Holowinsky and Soundararajan \cite{Holowinsky - Sound} using methods which are almost solely number theoretic. Their proof relied on some delicate analysis of mean values of arithmetic functions and many other results which are quite specific for this case such as the Watson's formula \cite{Watson - Dissertation}, the Ramanujan conjecture and a recent weak subconvexity result of Soundararajan \cite{Sound Weak Subconvexity}.

Even though QUE conjecture is now completely known for $ \mathfrak{X} $ (see \cite{Einsiedler notes}, \cite{Holowinsky} and \cite{Sound Escape of Mass} for the Maass form case) the known proofs crucially use the arithmetic nature of the surface $ \X $, and therefore a proof of Conjecture \ref{Conjecutre "QUE General"} in its complete generality is quite far at the moment. For example, the multiplicative nature of the Fourier coefficients are indispensable for the existing proofs. Hence it is unclear how to generalize the existing proofs, for example, when there is no action of a Hecke ring.

Therefore it is important to find proofs which makes a minimal use of the arithmetic of the surface $ \mathfrak{X} $, is amenable to generalizations for other situations, and hopefully shall one day lead to a complete proof of Conjecture \ref{Conjecutre "QUE General"}. The purpose of this paper is to provide such a proof. As we shall see, the present proof very minimally requires the multiplicativity of the Fourier coefficients and also does not require the full strength of the Ramanujan conjecture. Furthermore, the second half of the proof of Theorem \ref{Theorem "QUE holomorphic"}, i.e. Theorem \ref{Theorem "QUE Holomorphic Vertical"} is applicable for any modular form (including Maass forms and even half integral weight forms). However, it is unclear how to generalize the proof of Theorem \ref{Theorem "QUE Holomorphic Horizontal"} to Maass forms on $ \mathfrak{X} $ at the present moment. This lack of clarity is due to the oscillating behavior of the $ K $ - Bessel function as opposed to the incomplete $ \Gamma $ function whose behavior is somewhat simpler. Nevertheless the author believes that a modification of the present proof which accommodates the non-holomorphic case as well is not far away in the future (see \S \ref{Section "Description of the method"} and Remark \ref{Remark "Maass forms"}).

We proceed to describe the main results of this paper. Some standard references for the theory of integral and half integral weight modular forms are \cite{Cohen - Book} and \cite{Ono - Web of modularity}. We shall follow the notations of the following paragraphs throughout the paper.

First we shall consider holomorphic cuspidal modular forms the full modular group $ SL(2,\Z) $. For simplicity of exposition, we shall restrict ourselves to Hecke eigenforms of full level. Unless otherwise mentioned, every cusp form will be assumed to be of full level. For an even integer $ k $ we let $ f_k $ be a cuspidal Hecke eigenform of weight $ k $ for $ SL(2,\Z) $. We shall denote the $ n $th Fourier coefficient of $ f_k $ by $ a_k(n) $ and normalize $ f_k $ so that $ a_k(1)=1 $. Suppose that $ \psi $ is a smooth compactly supported test function on $ \mathfrak{X} $. For any $ k $ and a choice of $ f_k $ as above, denote by $ \mu_k $ the measure defined by
\begin{equation}\label{Equation "Measure definition holomorphic"}
	\mu_k(\psi) = \frac{1}{\|f_k\|^2} \int\limits_{\mathfrak{X}} y^k |f_k(z)|^2 \psi(z) \frac{dxdy}{y^2},
\end{equation}
where $ \|f_k\|^2 $ denotes the $ L^2 $ mass (Petersson Norm) of $ f_k $ and $ \psi\in C_c^\infty(\mathfrak{X}) $.

\begin{theorem}\label{Theorem "QUE holomorphic"}
	If $ \mu_k $ is as above and $ \psi\in C_c^\infty (\mathfrak{X}) $, then 
	\begin{equation}\label{Equation "QUE holomorphic theorem"}
		\lim_{k\to\infty} \mu_k(\psi) = \frac{3}{\pi} \int\limits_{\mathfrak{X}} \psi(z)\frac{dxdy}{y^2}.
	\end{equation}
	The limit holds independent of choice of $ f_k $'s and the rate of convergence depends on the support of $ \psi $, more precisely on the quantity $ \inf\{\Im(z)\ |\ z\in\mbox{Supp}(\psi)\} $.
\end{theorem}

We have the following corollaries to Theorem \ref{Theorem "QUE holomorphic"}.

\begin{corollary}\label{Corollary "Lehmer"}
	Suppose that $ f_k $'s are a sequence of Hecke eigenforms of weight $ k $, then for a given prime $ p $, there exists an integer $ k $ such that for all $ l > k $, we have $ a_l(p)\neq 0 $.
\end{corollary}

For a purely algebraic approach to Corollary \ref{Corollary "Lehmer"}, see \cite{Calegari}.

In light of \eqref{Equation "I_k(T) definition 2"} below, the problem of evaluating the limit \eqref{Equation "QUE holomorphic theorem"} can be easily transferred to the problem of estimating the partial sums of the form $ \sum_{n\leq X} \lambda^2_k(n) $, where $ \lambda^2_k(n) = a^2_k(n)n^{1-k} $. If $ X\gg k $, there is a very satisfactory asymptotic for these partial sums which can be deduced form the Rankin-Selberg arguments or from standard Perron's formula-Convexity bound arguments for the symmetric square $ L $ function associated to $ f_k $. However we are interested in the range $ X \asymp k $ and the above methods fail in this range. This is precisely because the spectral aspect convexity bound is of the size $ \sqrt{k} $ and this barely fails to provide us with a sharp enough bound. More precisely we require an arbitrarily small power saving of $ k $ (which is to say that we need an estimate of the type $ k^{1/2-\epsilon} $ for a small positive $ \epsilon $). If we use the bound of Hoffstein and Lockhart on $ L(1,\sym^2 f_k) $ and then the weak subconvexity bound of Soundararajan we again are short, this time by an arbitrarily small power of $ \log(k) $. Therefore, any improvements towards the `subconvexity' problem will automatically prove the QUE conjecture. But at the moment proving estimates of the type mentioned above seem extremely hard. We are however able to deduce the following corollary.

\begin{corollary}\label{Corollary "Mean Values"}
	If $ f_k $ is a Hecke eigenform of weight $ k $ for $ SL_2(\Z) $, then for any positive constant $ \epsilon $.
	\[
	\sum_{n\leq \epsilon k} |\lambda_k(n)|^2 = \epsilon kL(1,\sym^2 f_k) + o(kL(1,sym^2f_k)).
	\]
\end{corollary}

This can be compared with the error term arising from the weak subconvexity bound of Soundararajan \cite{Sound Weak Subconvexity}.

Suppose that $ F_k\in S_k(SL(2,\Z)) $ is a cusp form which is not necessarily a Hecke eigenform. Let us denote the basis of normalized Hecke eigenforms for $ S_k(SL(2,\Z)) $ by $ \{f_k^{(1)}, f_k^{(2)}, \ldots, f_k^{(j_k)}\} $. Furthermore suppose that $ F_k(z) = \sum_{i=1}^{j_k} \alpha_k^{(i)} f_k^{(i)}(z) $ for some complex numbers $ \alpha_k^{(i)} \neq 0 $. Since the Hecke eigenforms are orthogonal with respect to the Petersson inner product, we also see that 
\begin{equation}\label{Equation "Admissible norm"}
	\|F_k\|^2 = \sum_{i=1}^{j_k} |\alpha_k^{(i)}|^2 \|f_k^{(i)}(z)\|^2.
\end{equation}

\begin{definition}\label{Definition "Admissible forms"}
	A sequence of non zero cusp forms $ F_k $ (with $ \alpha_k^{(i)} $ as above) is called a sequence of \textit{admissible forms} if $ \min\{|\alpha_k^{(i)}|^2\} $ is bounded below independent of $ k $.
\end{definition}
In particular if $ \{F_k\} $ is a sequence of admissible forms, then for every fixed $ k $ it follows that 
\begin{equation}\label{Equation "Bound admissible form"}
	\|F_k\|^2 \geq |\alpha_k^{(i)}|^2\|f_k^{(i)}(z)\|^2 \gg \|f_k^{(i)}\|^2,
\end{equation}
where the implied constant is independent of $ k $. 

\begin{corollary}\label{Corollary "Rudnick"}
	If $ \{F_k\} $ is an admissible sequence of cusp forms, then for any test function $ \psi\in C_c^\infty(\X) $, we have
	\begin{equation}\label{Equation "QUE Admissible"}
		\lim_{k\to\infty} \frac{1}{\|F_k\|^2}  \int\limits_{\mathfrak{X}} y^k |F_k(z)|^2 \psi(z) \frac{dxdy}{y^2} = \frac{3}{\pi}  \int\limits_{\mathfrak{X}}\psi(z) \frac{dxdy}{y^2}.
	\end{equation}
\end{corollary}

Finally we observe the following result which can be interpreted from a quantum chaos perspective. The definition of the generalized Siegel domain $ \mathfrak{S}(a,b,T) $ is given in \eqref{Equation "Generalized Siegel domain definition"}.

\begin{corollary}\label{Corollary "Mass distribution holomorphic"}
	Given a holomorphic modular form $ f_k $ of weight $ k $ if $ T \geq 4 k\log(k) $ then 
	\[\mu_k\left(\mathfrak{S}\left(-\frac{1}{2},\frac{1}{2},T\right)\right)\leq \frac{e^{-2\pi T}}{2\pi T}.\]
\end{corollary} 

The paper is structured as follows. In \S \ref{Section "Description of the method"} we give a brief description of the method. In \S \ref{Section "Preliminaries"} we recall some basic notions and prove some preliminary lemmas which will be used later on. In \S \ref{Section "Proof of Theorem 1"} we prove Theorem \ref{Theorem "QUE holomorphic"}. In \S \ref{Section "Proof of Corollaries"} we shall prove the corollaries. Finally in \S \ref{Section "Concluding remarks"} we shall mention some remarks and hint at further directions leading from this work.

\subsection*{Notation}

\begin{enumerate}
	\item	For a complex number $ z $, we denote by $ \Re(z) $ and $ \Im(z) $, the real and imaginary parts of $ z $.
	\item	$ dz $ denotes the measure $ dxdy/y^2 $.
	\item	$ k $ will denote an integer.
	\item	The symbol $ f_k $ will denote a cuspidal eigenform of integral weight $ k $ normalized so that $ a_k(1)=1 $ where $ a_k(n) $ is the $ n $th Fourier coefficient of $ f_k $.
	\item	$ \langle \cdot, \cdot \rangle_\X $ denotes the Petersson inner product on $ \mathfrak{X} $. The weight will be clear from the context.
	\item	We let $ \tau(n) $ denote the number of divisors of $ n $.
	\item	The symbol $ \epsilon $ shall denote an arbitrarily small positive quantity which may not be equal in different occurrences.
\end{enumerate}

\subsection*{Acknowledgments}

The author wishes to thank Prof. Kalyan Chakraborty for his support and encouragement and the Kerala School of Mathematics for its generous hospitality. The author would like to thank Mr. Sreejith M. M.\footnote{Kerala School of Mathematics}, Dr. Pramath Anamby\footnote{Harish Chandra Research Institute} and Prof. Ben Kane\footnote{Hong Kong University.} for wonderful discussions. The author also wishes to thank Prof. Ze\'ev Rudnick\footnote{Tel-Aviv University.}, Prof. Maksym Radziwill\footnote{Caltech.} and Prof. Valentin Blomer\footnote{Universit\"at Bonn.} for their valuable comments on the first draft of this manuscript.

\section{Description of the Method}\label{Section "Description of the method"}

In this section we shall give a brief description of the method without going into too much details. We first make some reductions to what is to be proven in terms of Theorem \ref{Theorem "QUE holomorphic"}. First we observe that it is enough to prove \eqref{Equation "QUE holomorphic theorem"} for a dense subset of $ C_c^\infty(\mathfrak{X}) $. In fact the complementary approaches of Soundararajan and Holowinsky in their proofs fundamentally differed in this choice of a dense subset (see \cite{Sound Notes} for an exposition of their results).

We choose the characteristic functions of rectangles (see \eqref{Equation "Recatngle definition"}) as our basis for $ C_c^\infty(\mathfrak{X}) $. Strictly speaking, the characteristic functions that we shall consider are neither smooth nor compactly supported, but they form a basis of $ L^1(\mathfrak{X}) $. It is in this sense that we say that the characteristic functions of rectangles form a basis for $ C_c^\infty (\mathfrak{X}) $. The advantage of making such a choice of basis is that it allows us to separate out the distribution of mass in the ``horizontal" and ``vertical" directions and study them separately. First we show the horizontal equidistribution of mass and deduce the vertical equidistribution from this.

Suppose we write $ z=x+iy $, where $ x,y\in\R $. We crucially use the fact that the modular forms considered can be written as a Fourier series involving certain coefficients summed against certain ``kernels". These kernels naturally split into two parts. The first part is periodic in $ x $ with period $ 1 $ and the second part has exponential decay in $ y $ as $ y\to\infty $. Corollary \ref{Corollary "Mass distribution holomorphic"} is a consequence of this exponential decay. The process of letting $ k\to\infty $ slows down the exponential decay of the kernel functions and it is precisely this slowing down which changes the behavior of the limit from the one described in Corollary \ref{Corollary "Mass distribution holomorphic"} to the one described in Theorem \ref{Theorem "QUE holomorphic"}.

As we have noted above, improvements towards the subconvexity problem for certain $ L $ functions will automatically prove the QUE conjecture. In fact using the Watson's formula, we can show that arbitrary power saving towards the subconvexity problem of $ L(1/2, f_k\times f_k\times \phi_j) $ (where $ f_k $ is as above and $ \phi_j $ is a Maass form) will imply QUE. Furthermore, following the arguments in \S \ref{Section "Mean Values in short intervals"}, we see that arbitrary power saving towards the subconvexity bound for $ L(1/2, \sym^2 f_k) $ is sufficient for the truth of QUE. Our method reudces this a step further. We shall deduce QUE from a certain weaker form of the Ramanujan bound (see Remark \ref{Remark "Ramanujan"}).

Much stronger and deeper results than what we require are available in the literature but in almost all of the cases the emphasis is on improving the exponent of $ n $ in the case of the bound on the Fourier coefficient. These power savings are achieved in many cases at the expense of a larger power of $ k $ occurring as an implicit constant. Therefore the uniformity of the bounds is as important to us as the strength of the bound itself.

Of course, all of these bounds are required to prove the horizontal equidistribution. Once the horizontal equidistribution is known, quite surprisingly we are able to deduce the vertical equidistribution and Theorem \ref{Theorem "QUE holomorphic"} by using nothing more than the modularity of $ f_k $'s and some basic calculus. 

\section{Preliminary Results and Background}\label{Section "Preliminaries"}

\subsection{Review of Quantum Unique Ergodicity}

There is a considerable amount of literature available today regarding the history and relevance of Conjecture \ref{Conjecutre "QUE General"}, some of which are \cite[Chapter 9]{Bergeron - Book}, \cite{Einsiedler notes}, \cite{Sarnak - Notes} \cite{Sound Notes}, \cite{Sound Escape of Mass}. One may also refer the introductory sections of \cite{Lindenstrauss} and \cite{Nelson - QUE varying level}. Because of the existence of such abundant literature, we shall only briefly mention the background for Conjecture \ref{Conjecutre "QUE General"} for the sake of completeness and for the convenience of the reader. For simplicity we shall restrict ourselves to $ \mathfrak{X} $ and describe the motivation from the perspective of Maass forms.

It can be shown that the geodesic flow on $ \X $ which can be considered as a Hamiltonian flow is ergodic, this is to say that every measurable set of $ \X $ invariant under the geodesic flow either has full measure or zero measure. Another interpretation which can be given is that the path of a ``classical particle" which is moving under the geodesic flow in $ \X $ will equidistribute as the time goes to infinity. This is the case of the classical geodesic flow.

In analogy of the reformulation of classical mechanics into quantum mechanics, one can reformulate the above phenomenon in terms of a ``quantum" flow. In this case we shall consider the flow of a ``quantum particle". In this case the various states (referred to as eigenstates) of the particle are replaced by the various $ L^2 $ normalized eigenfunctions of the hyperbolic Laplacian $ \Delta $. Following Born's interpretation the size of the eigenvalue then becomes a measure of the energy and the probability measure associated to the eigenfunction by Riesz representation theorem becomes the probability of finding that particle inside any measurable set. Therefore it makes sense to ask the analogous question of the classical ergodicity of geodesic flow mentioned in the previous paragraph in the quantum setting in the semiclassical limit. This is precisely Conjecture \ref{Conjecutre "QUE General"}. In this case, instead of letting the time go to infinity, we let the energy of the particle go to infinity, and ask the question whether the corresponding probability measures equidistribute with respect to the usual hyperbolic measure on $ \X $.

A weaker result, often called ``Quantum Ergodicity" (which was proven by Zelditch \cite{Zelditch - QE} in a more general setup) asserts that the mass of a typical eigenstate equidistributes as the eigenvalue goes to infinity. More precisely there is a density $ 1 $ subsequence of the eigenvalues $ \lambda_i $'s (as in Conjecture \ref{Conjecutre "QUE General"}) such that the corresponding probability measures equidistribute in the sense of Conjecture \ref{Conjecutre "QUE General"}. Therefore quantum ergodicity asserts that Conjecture \ref{Conjecutre "QUE General"} is true modulo the existence of an exceptional subsequence. Quantum `unique' ergodicity now asserts that such an exceptional subsequence cannot exist.

\subsection{Generalized Siegel domains} Let $ \mathfrak{F} $ be the standard closed fundamental domain for the action of $ SL(2,\Z) $ on the upper half plane $ \h $, that is
\[
\mathfrak{F} := \left\{ z = x+iy\ \Big|\ -\frac{1}{2} \leq x \leq \frac{1}{2}, |z| \geq 1\right\}.
\]
We shall interchangeably identify $ \mathfrak{F} $ with $ \mathfrak{X} $ and as the subset of $ \h $. Denote by $ \Gamma_\infty $ the subgroup of translations in $ SL(2,\Z) $ and denote the standard fundamental domain for the action of $ \Gamma_\infty $ on $ \h $ by $ \mathfrak{T} $. That is
\[
\mathfrak{T} = \left\{z = x+iy\ \Big|\ -\frac{1}{2} < x \leq \frac{1}{2}, y>0\right\}.
\]

Suppose that $ z_1,z_2\in \h/ \Gamma_\infty $ such that $ \Re(z_1)\neq \Re(z_2) $ and $ \Im(z_1)\neq \Im(z_2) $. Without loss of generality, assume that $ \Re(z_1) < \Re(z_2) $ and $ \Im(z_1) < \Im(z_2) $. Define 
\begin{equation}\label{Equation "Recatngle definition"}
	\mathcal{R}(z_1,z_2) : = \{z = x+iy\ |\ \Re(z_1) < x < \Re(z_2),\ \Im(z_1) < y < \Im(z_2) \}.
\end{equation}
For obvious reasons, sets of the form $ \mathcal{R}(z_1,z_2) $ will be called rectangles. The conditions on $ z_1, z_2 $ make sure that the rectangles have positive volume with respect to the hyperbolic measure $ dz $. 

For $ -\frac{1}{2} \leq a < b \leq \frac{1}{2} $ and $ T>0 $, define 
\begin{equation}\label{Equation "Generalized Siegel domain definition"}
	\mathfrak{S}(a,b,T) := \{z = x+iy\ |\ a<x<b, y>T\} \subset \mathfrak{T}.
\end{equation}

\subsection{Some preparatory results}

We start with some technical lemmas.

\begin{lemma}\label{Lemma "Gamma function"}
	Suppose $ \delta > 0 $, then 
	\[
	\lim_{k\to\infty} \frac{\Gamma(k-1, k - k^{1/2 + \delta})}{(k-2)!} = 1.
	\]
\end{lemma}

\begin{proof}
	We shall assume that $ k $ is large enough and for the time being fix an $ \epsilon $ such that $ 0\leq \epsilon <1 $. We start with the following well known expression for the incomplete $ \Gamma $ function valid for integer values of $ k $ \cite[6.5.13]{Abramovich - Stegun},
	\begin{equation}\label{Equation "Incomplete Gamma formula"}
		\Gamma(k,x) = (k-1)! e^{-x} \sum_{m=0}^{k-1} \frac{x^m}{m!}.
	\end{equation}
	This gives us that 
	\begin{align}\nonumber
		\Gamma(k-1, \epsilon k) &= (k-2)! e^{-\epsilon k} \sum_{m=0}^{k-2} \frac{(\epsilon k)^m}{m!}\\ \nonumber
		\Rightarrow \frac{\Gamma(k-1,\epsilon k)}{(k-2)!}&= e^{-\epsilon k} \sum_{m=0}^{k-2} \frac{(\epsilon k)^m}{m!}\\
		&= 1 - e^{-\epsilon k} \sum_{m=k-1}^{\infty}  \frac{(\epsilon k)^m}{m!}. \label{Equation "Gamma bound 1"}
	\end{align}
	Observe that, as long as $ m \geq k-1 $, we have
	\begin{equation}\label{Equation "k^m/m!"}
		\frac{(\epsilon k)^{m+1}}{(m+1)!} = \frac{(\epsilon k)^m}{m!} \frac{\epsilon k}{m+1} \leq \frac{(\epsilon k)^m}{m!}.
	\end{equation}
	Therefore,
	\begin{align*}
		e^{-\epsilon k}\sum_{m=k-1}^{\infty}  \frac{(\epsilon k)^m}{m!} &< e^{-\epsilon k} \frac{(\epsilon k)^{k-1}}{(k-1)!} \sum_{m=0}^{\infty} \frac{(\epsilon k)^m}{\prod_{j=0}^{m} (k+j)}\\
		& < e^{-\epsilon k} \frac{(\epsilon k)^{k-1}}{(k-1)!}  \sum_{m=0}^{\infty} \epsilon^m\\
		&= \left(e^{-\epsilon k} \frac{(\epsilon k)^{k-1}}{(k-1)!}\right) \left(\frac{1}{1- \epsilon}\right)\\
		&= \left(\frac{e^{-k} k^k}{k!}\right)\left(\frac{e^{-\epsilon k}\epsilon^k}{e^{-k}(1- \epsilon)\epsilon}\right).
	\end{align*}
	Now given a choice of $ 1/2 > \delta > 0 $, if we substitute $ \epsilon = 1- k^{-1/2 + \delta}$, then using the Stirling's approximation, the right hand side of the above equality is bounded above by $ k^{-\delta} $ upto a constant independent of $ k $. This completes the proof of the lemma.
\end{proof}

The following lemma is an easy variant of the integral test for convergence of series.

\begin{lemma}\label{Lemma "Lemma Sum Integral"}
	Suppose that $ a(n) $ is a summable sequence of positive real numbers such that $ a(n+1)/a(n) $ is a decreasing sequence. Suppose that $ f:(0,\infty)\to(0,\infty) $ is a continuous function such that $ f(n) = a(n) $ such that $ f $ is increasing on $ (0,\tau) $ and decreasing on $ (\tau,\infty) $ for some $ \tau\in(0,\infty) $. Then for an integer $ k $, we have
	\[
	\sum_{n=k+1}^{\infty} a(n) \leq \max\left\{\frac{a(k+1)}{a(k)},1\right\} \int\limits_{k}^{\infty} f(x) dx.
	\]
\end{lemma}

\begin{proof}
	If $ k> \tau $, then the lemma is just the integral test for series convergence. Therefore suppose that $ k<\tau $. It follows that, if $ k<  l < \tau $ then 
	\[
	a(l) \leq \frac{a(l)}{a(l-1)}\int\limits_{l-1}^{l} f(x)dx \leq \frac{a(k+1)}{a(k)}\int\limits_{l-1}^{l} f(x)dx.
	\]
	Therefore we have
	\begin{align*}
		\sum_{n=k+1}^{\infty} a(n) &= \sum_{n\leq \tau}a(n) + \sum_{n \geq \tau} a(n)\\
		&\leq \frac{a(k+1)}{a(k)} \int\limits_{k+1}^{\tau} f(x) dx + \int\limits_{\tau}^\infty f(x)dx\\
		&\leq \max\left\{\frac{a(k+1)}{a(k)},1\right\} \int\limits_{k}^{\infty} f(x) dx
	\end{align*}
	as claimed.
\end{proof}

\begin{remark}
	Our primary use of Lemma \ref{Lemma "Lemma Sum Integral"} will be to estimate the sum $ \sum_{n=m}^\infty n^k e^{-\alpha n} $ for some positive constant $ \alpha $ and an integer $ k $. We remark here that Lemma \ref{Lemma "Lemma Sum Integral"} is applicable in this case.
\end{remark}

Define $ \lambda^2_k(n) := a^2_k(n)n^{1-k}$ as above and recall that the symmetric square $ L $ function associated to a cusp form $ f_k $ is given by 
\begin{equation}\label{Equation "Definition Symmetric square L function"}
	L(s, \sym^2(f_k)) = \sum_{n=1}^{\infty} \frac{\lambda_k^2(n)}{n^s}.
\end{equation}
The function $ L(\sym^2(f_k),s) $ is analytic in the half plane $ \Re(s) > 1 $, has a simple pole at the point $ s=1 $ with a residue equal to $ (3/\pi) ((4\pi)^k/(k-1)!) \|f_k\|^2 $, continues to a meromorphic function to the entire complex plane and satisfies a functional equation connecting the values of $ s $ and $ 1-s $. If $ f_k $ is a newform, we have the following bound due to Hoffstein and Lockhart \cite{Hoffstein - Lockhart}.

\begin{equation}\label{Equation "Bound f_k"}
	\|f_k\|^2 = L(1,\sym^2f_k) \frac{2}{\pi} \frac{(k-1)!}{(4\pi)^k} \gg \frac{2}{\pi} \frac{(k-1)!}{(4\pi)^k \log(k)}.
\end{equation}

\section{Proof of Theorem \ref{Theorem "QUE holomorphic"}}\label{Section "Proof of Theorem 1"}

Suppose we let 
\begin{equation}\label{Equation "I_k(T) definition"}
I_k(T) := \frac{1}{\|f_k\|^2}\int\limits_{T}^{\infty} \int\limits_{-\frac{1}{2}}^{\frac{1}{2}}  y^k|f_k(z)|^2 \frac{dxdy}{y^2} = \frac{1}{\|f_k\|^2} \int\limits_{-\frac{1}{2}}^{\frac{1}{2}}\int\limits_{T}^{\infty}  y^k|f_k(z)|^2 \frac{dydx}{y^2}.
\end{equation}	
By Parseval's formula we have,
\begin{align*}
I_k(T) &=\frac{1}{\|f_k\|^2} \int\limits_{T}^{\infty} y^{k} \sum_{n=1}^{\infty} |a_k(n)|^2 e^{-4\pi ny} \frac{dy}{y^2}\\
&= \frac{1}{\|f_k\|^2}\sum_{n=1}^{\infty} |a_k(n)|^2 \int\limits_{T}^{\infty} y^{k} e^{-4\pi ny} \frac{dy}{y^2}\\
&= \frac{1}{\|f_k\|^2}\frac{1}{(4\pi)^{k-1}}\sum_{n=1}^{\infty} \frac{|a_k(n)|^2}{n^{k-1}} \int\limits_{4\pi nT}^{\infty} w^{k} e^{-w} \frac{dw}{w^2},
\end{align*}
after a change of variable.
If we observe that the integral is the incomplete $ \Gamma $ function, we can rewrite the above equation as 
\begin{equation}\label{Equation "I_k(T) definition 2"}
I_k(T) = \frac{1}{\|f_k\|^2}\frac{1}{(4\pi)^{k-1}} \sum_{n=1}^{\infty} |\lambda_k(n)|^2 \Gamma(k-1, 4\pi nT).
\end{equation}
We observe here that $ \lim_{T\to\infty} I_k(T) = 0 $.
First we prove the horizontal equidistribution mentioned in \S \ref{Section "Description of the method"} in Theorem \ref{Theorem "QUE Holomorphic Horizontal"} below.

\begin{theorem}\label{Theorem "QUE Holomorphic Horizontal"}
	For real numbers $ a,b $ such that $ a < b $, we have
	\[
	\lim_{k\to\infty} \frac{1}{\|f_k\|^2} \int\limits_{T}^{\infty} \int\limits_{a}^{b}  y^k|f_k(z)|^2 \frac{dxdy}{y^2} = (b-a) \lim_{k\to\infty} I_k(T)
	\]
	The above statement is true independent of choice of $ f_k $'s and the rate of convergence depends on $ T $.
\end{theorem}

\begin{proof}
	Let $ a < b $ be two real numbers. Fix $ y > 0 $ and consider the integral 
	\[
	\int\limits_{a}^{b} |f_k(z)|^2 dx.
	\]
	
	We have
	\begin{align*}
		\int\limits_{a}^{b}& |f_k(z)|^2 dx = (b-a) \sum_{n=1}^{\infty} a_k^2(n) e^{-4\pi ny} + \sum_{n\neq m} a_k(n)a_k(m)e^{-2\pi(n+m)y}\int\limits_{a}^{b} e^{2\pi i(n-m)x}dx\\
		&=(b-a) \sum_{n=1}^{\infty} a_k^2(n) e^{-4\pi ny} + \sum_{n > m\geq 1} a_k(n)a_k(m)e^{-2\pi(n+m)y}\left(\int\limits_{a}^{b} e^{2\pi i(n-m)x}dx + \int\limits_{a}^{b} e^{2\pi i(m-n)x}dx\right)\\
		&=(b-a) \sum_{n=1}^{\infty} a_k^2(n) e^{-4\pi ny} + \sum_{n > m\geq 1} a_k(n)a_k(m)e^{-2\pi(n+m)y}\left(\int\limits_{a}^{b}\cos(2\pi (n-m)x)dx\right).
	\end{align*}
	In light of \eqref{Equation "I_k(T) definition 2"} the first sum should be thought as the ``main term" and the second sum should be thought of as the ``error term". Observe that the proof of theorem will be complete if we show the second sum decays to zero with $ k $. Let us consider the second summation in the previous equation.
	\begin{align*}
		&\left|\sum_{n > m \geq 1} a_k(n)a_k(m)e^{-2\pi(n+m)y}  \left(\int\limits_{a}^{b}\cos(2\pi (n-m)x)dx\right)\right|\\
		&\ll \sum_{n > m \geq 1} \left|\frac{a_k(n)a_k(m)e^{-2\pi(n+m)y}}{2\pi (n-m)} \right|.
	\end{align*}

Suppose we bound $ |a_k(n)| $ by $ n^\delta $. Then the above summation can be rewritten as
\begin{align*}
	\sum_{n=1}^{\infty} n^\delta e^{-4\pi ny}&\sum_{l=1}^{\infty} \frac{(n+l)^{\delta}}{l} e^{-2\pi ly} = \sum_{n=1}^{\infty} n^\delta e^{-4\pi ny}\sum_{l=1}^{\infty} (n+l)^{\delta-1}\left(1 + \frac{n}{l}\right) e^{-2\pi ly}\\
	&=\sum_{n=1}^{\infty} n^\delta e^{-4\pi ny}\sum_{l=1}^{\infty} (n+l)^{\delta-1} e^{-2\pi ly} + \sum_{n=1}^{\infty} n^{\delta+1} e^{-4\pi ny}\sum_{l=1}^{\infty} \frac{(n+l)^{\delta-1}}{l} e^{-2\pi ly}.
\end{align*}
Simplifying the second term similarly we get that
\begin{multline}\label{Equation "Tail horizontal"}
	\sum_{n=1}^{\infty} n^\delta e^{-4\pi ny}\sum_{l=1}^{\infty} \frac{(n+l)^{\delta}}{l} e^{-2\pi ly} = \\
	\sum_{j=0}^{r} \sum_{n=1}^{\infty} n^{\delta+j} e^{-4\pi ny}\sum_{l=1}^{\infty} (n+l)^{\delta-1-j} e^{-2\pi ly} + \sum_{n=1}^{\infty} n^{\delta+1+r} e^{-4\pi ny}\sum_{l=1}^{\infty} \frac{(n+l)^{\delta-1-r}}{l} e^{-2\pi ly},
\end{multline}
for any integer $ r \geq 0 $. Considering the first term, we have
\begin{align*}
	\sum_{j=0}^{r} \sum_{n=1}^{\infty} n^{\delta+j} e^{-4\pi ny}\sum_{l=1}^{\infty} (n+l)^{\delta-1-j} e^{-2\pi ly}&= \sum_{n=1}^{\infty} n^{\delta} e^{-4\pi ny}\sum_{l=1}^{\infty} (n+l)^{\delta-1} e^{-2\pi ly} \sum_{j=0}^{r} \left(\frac{n}{n+l}\right)^j\\
	&\leq \sum_{n=1}^{\infty} n^{\delta-1} e^{-4\pi ny}\sum_{l=1}^{\infty} (n+l)^{\delta} e^{-2\pi ly}\\
	&\ll \frac{1}{(2\pi y)^{2\delta-1}}\Gamma(\delta)\Gamma(\delta+1)
\end{align*}

To evaluate the second term, if $ l \geq \log(k)/2\pi $, then we see that
\[
\sum_{n=1}^{\infty} n^{\delta+1+r} e^{-4\pi ny}\sum_{l=\frac{\log(k)}{2}}^{\infty} \frac{(n+l)^{\delta-1-r}}{l} e^{-2\pi ly} \ll \frac{e^{-\log(k) y}}{\log(k) (4\pi y)^k}\Gamma(\delta + r + 2).
\]

Suppose we choose $ r = k/2 $ and $ \delta = k/2-1/2+\epsilon $, we have $ \delta - 1 - r = -1/2+\epsilon $ and $ \delta + 1 + r = k-1/2+\epsilon $. With these choices of $ r,\delta $, the right hand side of \eqref{Equation "Tail horizontal"}, for $ l\geq \log(k)/2\pi $ is bounded above by
\[
\ll \frac{\Gamma\left(\frac{k-1}{2} + \epsilon\right)\Gamma\left(\frac{k+1}{2} + \epsilon\right)}{(2\pi y)^k} + \frac{e^{-\log(k) y}}{\log(k) (4\pi y)^k}\Gamma\left(k+\frac{1}{2}+\epsilon\right).
\]
Integrating with respect to $ y^{k-2}dy $, dividing by $ \|f_k\|^2 $ and using \eqref{Equation "Bound f_k"}, we can bound the right hand side above by
\begin{equation}\label{Equation "Horizontal Tail bound 1"}
	\ll_T \frac{2^k\Gamma\left(\frac{k-1}{2} + \epsilon\right)\Gamma\left(\frac{k+1}{2} + \epsilon\right)}{\Gamma(k)} + \frac{e^{- \log(k)}}{\log(k)}\frac{\Gamma\left(k+\frac{1}{2}+\epsilon\right)}{\Gamma(k)},
\end{equation}
which clearly goes to zero with $ k $. We are left to bound the sum
\[
\left|\sum_{n=1}^\infty a_k(n)e^{-4\pi ny}\sum_{l=1}^{\log(k)/2\pi} a_k(n+l) e^{-2\pi ly}  \left(\int\limits_{a}^{b}\cos(2\pi lx)dx\right)\right|.
\]
We have
\begin{align*}
	\sum_{n=1}^{\infty}& a_k(n)e^{-4\pi ny}\sum_{l=1}^{\log(k)/2\pi} a_k(n+l) e^{-2\pi ly}  \left(\int\limits_{a}^{b}\cos(2\pi lx)dx\right) \\
	&= \sum_{l=1}^{\log(k)/2\pi} e^{-2\pi ly}\left(\int\limits_{a}^{b}\cos(2\pi lx)dx\right) \sum_{n=1}^{\infty} a_k(n) a_k(n+l)e^{-4\pi ny}\\
	&\ll \sum_{l=1}^{\log(k)/2\pi} \frac{e^{-2\pi ly}}{l}\left|\sum_{n=1}^{\infty} a_k(n) a_k(n+l)e^{-4\pi ny}\right|.
\end{align*}
In order to truncate the sum further, fix a small $ \epsilon > 0 $ and consider
\[
\left|\sum_{n = k + k^{\frac{1}{2}+\epsilon}}^{\infty} a_k(n) a_k(n+l)e^{-4\pi ny}\right| \leq \sum_{n = k + k^{\frac{1}{2}+\epsilon}}^{\infty} |a_k(n) a_k(n+l)| e^{-4\pi ny}.
\]
Integrating with respect to $ y^{k-2} dy$ and dividing by $ \|f_k\|^2 $ we are left with the expression
\begin{align*}
	\frac{1}{\|f_k\|^2} \frac{1}{(4\pi)^{k-1}} &\sum_{n = k + k^{\frac{1}{2}+\epsilon}}^{\infty} \frac{|a_k(n) a_k(n+l)|}{n^{k-1}} \Gamma(k-1, 4\pi nT)\\
	&\ll \frac{1}{\|f_k\|^2} \frac{1}{(4\pi)^{k-1}} \sum_{n = k + k^{\frac{1}{2}+\epsilon}}^{\infty} n^{2\epsilon} \left(1 + \frac{l}{n}\right)^{\frac{k-1}{2}+\epsilon} \Gamma(k-1, 4\pi nT)\\
	&\ll \left(1+\frac{\log(k)}{2(k+k^{\frac{1}{2}+\epsilon})}\right)^{\frac{k-1}{2}+\epsilon} \frac{1}{\|f_k\|^2} \frac{1}{(4\pi)^{k-1}} \sum_{n = k + k^{\frac{1}{2}+\epsilon}}^{\infty} n^{2\epsilon} \Gamma(k-1, 4\pi nT).
\end{align*}
Now the right hand side can be shown to go to zero similar to how $ E_k(T) $ (see \eqref{Equation "S_2 definition"}) is shown to vanish in \S \ref{Section "Mean Values in short intervals"} below. Similarly we have
\begin{multline}\label{Equation "Limiting range"}
	\left|\sum_{n = k - k^{\frac{3}{4}+\epsilon}}^{k + k^{\frac{1}{2}+\epsilon}} a_k(n) a_k(n+l)e^{-4\pi ny}\right|\\
	\ll \left(1+\frac{\log(k)}{2\pi(k-k^{\frac{1}{2}+\epsilon})}\right)^{\frac{k-1}{2}+\epsilon} \frac{1}{\|f_k\|^2} \frac{1}{(4\pi)^{k-1}} \sum_{n = k - k^{\frac{3}{4}+\epsilon}}^{k + k^{\frac{1}{2}+\epsilon}} n^{2\epsilon} \Gamma(k-1, 4\pi nT)
\end{multline}
which goes to zero with $ k $. This leaves us with the sum 
\[
\sum_{l=1}^{\log(k)/2\pi} e^{-2\pi ly}\left(\int\limits_{a}^{b}\cos(2\pi lx)dx\right) \sum_{n=1}^{k - k^{\frac{3}{4}+\epsilon}} a_k(n) a_k(n+l)e^{-4\pi ny}.
\]
Applying the Ramanujan bound we can bound this sum by
\[
\sum_{l=1}^{\frac{\log(k)}{2\pi}} \frac{e^{-2\pi ly}}{l} \sum_{n=1}^{k-k^{\frac{3}{4}+\epsilon}} n^{2\delta} \left(1+\frac{l}{n}\right)^\delta e^{-4\pi ny}.
\]
If $ n \geq c_1 k $ for some constant $ c_1 $, then we have 
\begin{align*}
	\sum_{l=1}^{\frac{\log(k)}{2\pi}} \frac{e^{-2\pi ly}}{l} \sum_{n= c_1 k}^{k-k^{\frac{3}{4}+\epsilon}} n^{2\delta} \left(1+\frac{l}{n}\right)^\delta e^{-4\pi ny} &\ll \sum_{n= c_1 k}^{k-k^{\frac{3}{4}+\epsilon}} n^{2\delta} \left(1+\frac{\log(k)}{2\pi n}\right)^\delta e^{-4\pi ny}\\
	&\ll k^{\frac{1}{4\pi c_1}}\sum_{n=c_1k}^{k-k^{\frac{3}{4}+\epsilon}} n^{2\delta} e^{-4\pi ny}\\
	&\ll_T k^{\frac{1}{4\pi c_1}}\frac{\gamma(k+\epsilon, k-k^{\frac{3}{4}+\epsilon})}{\Gamma(k)}\\
	&\ll_T k^{\frac{1}{4\pi c_1} - \frac{1}{4}+\epsilon},
\end{align*}
where the second to last inequality is obtained after integrating with respect to $ y^{k-2}dy $ and dividing by $ \|f_k\|^2 $. Choosing $ c_1 $ close enough to $ 1 $ we see that the right hand side goes to zero with $ k $. If $ c_2 \log(k) \leq n\leq c_1k $, then we have
\begin{align*}
	\sum_{n= c_2\log(k)}^{c_1 k} n^{2\delta} \left(1+\frac{\log(k)}{2n}\right)^\delta e^{-4\pi ny} &\ll \left(1 + \frac{1}{2c_2}\right)^\delta \sum_{n= c_2\log(k)}^{c_1 k} n^{2\delta} e^{-4\pi ny}\\
	&\ll_T \left(1 + \frac{1}{2c_2}\right)^\delta \left(\frac{e^{-c_1 k}c_1^k}{e^{-k}(1- c_1)c_1}\right) p(k)\\
	&\ll_T \left(1 + \frac{1}{2c_2}\right)^\delta \frac{(e^{1-c_1}c_1)^k}{(1-c_1)c_1} p(k)
\end{align*}
where again the second inequality is obtained after integrating with respect to $ y^{k-2} dy $ and dividing by $ \|f_k\|^2 $, and $ p(k) $ is a polynomial on $ k $. For a given value of $ c_1 $, choosing $ c_2 $ large enough, we see that the right hand side has exponential decay in $ k $.

This leaves us with the sum 
\begin{align*}
	\sum_{l=1}^{\frac{\log(k)}{2}} \frac{e^{-2\pi ly}}{l} \sum_{n= 1}^{c_2\log(k)} n^{2\delta} \left(1+\frac{l}{n}\right)^\delta e^{-4\pi ny} &\ll \sum_{l=1}^{\frac{\log(k)}{2}} \frac{(1+l)^\delta e^{-2\pi ly}}{l} \sum_{n= 1}^{c_2\log(k)} n^{2\delta} e^{-4\pi ny}\\
	&\ll \frac{\gamma(\delta+1, \log(k)/2) \gamma(2\delta + 1, c_2 \log(k)) }{(2\pi y)^{\delta +1} (4\pi y)^{2\delta +1}}\\
	&\ll_T \frac{\gamma(\delta+1, \log(k)/2) \gamma(2\delta + 1, c_2 \log(k))}{(2\pi)^{\delta} \Gamma(k)},
\end{align*}
where the right hand side clearly goes to zero with $ k $.

This completes the proof of Theorem \ref{Theorem "QUE Holomorphic Horizontal"}.

\end{proof}

Next we shall prove the vertical equidistribution in Theorem \ref{Theorem "QUE Holomorphic Vertical"} below. It can also be thought of as the analogue of the ``no escape of mass" result of Soundararajan \cite{Sound Escape of Mass}.

\begin{theorem}\label{Theorem "QUE Holomorphic Vertical"}
	For any $ T > 0 $, we have 
	\[
	\lim_{k\to\infty} \frac{1}{\|f_k\|^2} \int\limits_{T}^{\infty} \int\limits_{-\frac{1}{2}}^{\frac{1}{2}}  y^k|f_k(z)|^2 \frac{dxdy}{y^2} = \frac{3}{\pi T}.
	\]
	The above statement is true independent of choice of $ f_k $'s and the rate of convergence depends on $ T $.
\end{theorem}

\begin{proof}
	From Theorem \ref{Theorem "QUE Holomorphic Horizontal"} above, we can write
	\[
	\frac{1}{\|f_k\|^2} \int\limits_{a}^{b} \int\limits_{T}^{\infty}  y^k|f_k(z)|^2 \frac{dydx}{y^2} = (b-a) \frac{1}{\|f_k\|^2}  \int\limits_{-1/2}^{1/2} \int\limits_{T}^{\infty} y^k|f_k(z)|^2 \frac{dydx}{y^2} + o(1).
	\]
	For large enough $ k $, we can write
	\begin{equation}\label{Equation "I_k(T) definition alternate"}
		I_k(T) = \frac{1}{\|f_k\|^2}\int\limits_{T}^\infty |f_k(x+iy)|^2 y^{k-2} dy + o(1).
	\end{equation}
	Observe that \eqref{Equation "I_k(T) definition alternate"} is valid for every $ x\in\R $ (strictly speaking the error term depends on $ x $ as well but since $ f_k $ is periodic with resepct to $ x $ and the possible inequivalent choices for $ x $ is bounded, therefore we omit that dependence). We shall consider the integral in \eqref{Equation "I_k(T) definition alternate"}, more precisely we shall consider its derivative. To that end on analyzing the error term occuring in \eqref{Equation "I_k(T) definition alternate"}, the derivative of the error term in \eqref{Equation "I_k(T) definition alternate"} with respect to $ T $ is given by
	\[
	\sum_{n > m\geq 1} a_k(n)a_k(m)e^{-2\pi(n+m)T}\left(\int\limits_{a}^{b}\cos(2\pi (n-m)x)dx\right),
	\] 
	which can be bounded along the same lines as the bounding of the error term in the proof of Theorem \ref{Theorem "QUE Holomorphic Horizontal"}.
	
	Therefore we can write
	\[
	I_k'(T) = - \frac{|f_k(x+iT)|^2T^{k-2}}{\|f_k\|^2} + o(1).
	\]
	From the modularity of $ f_k $, we have for every $ \gamma = \begin{pmatrix} a&b\\c&d\end{pmatrix}\in SL_2(\Z) $, 
	\begin{equation}\label{Equation "I_k' modularity"}
		I_k'\left(\Im(\gamma (iT))\right) = (c^2T^2 + d^2)^2 I_k'(T) + o(1).
	\end{equation}
	Now choose two primes $ p,q \equiv 1\mod 4 $ and consider their quotient $ p/q $. The set of all such quotients is dense in $ \R $ (see \cite{Micholson}). Choosing $ \gamma\in SL_2(\Z) $ such that the bottom row equals $ (q,p) $ we get
	\[
	I_k'\left(\frac{p}{q}\right) = \frac{1}{(2p^2)^2} I_k'\left(\frac{1}{2pq}\right).
	\]
	From the choice of $ p,q $, we can find two coprime integers $ c,d $ such that $ c^2+d^2 = 2pq $. Using \eqref{Equation "I_k' modularity"} again with $ T=1 $ and $ c,d $ as above we see that
	\[
	I_k'\left(\frac{p}{q}\right) = \frac{q^2}{p^2} I_k'(1) + o(1).
	\]
	From the continuity of $ I_k' $ we have
	\[
	I_k'(T) = \frac{I_k'(1)}{T^2} + o(1).
	\]
	Integrating we see that 
	\[
	I_k(T) = -\frac{I_k'(1)}{T} + c_k + o(1).
	\]
	On close examination of the proof of Theorem \ref{Theorem "QUE Holomorphic Horizontal"}, we see that the error term is inversely proportional to $ T $. Therefore letting $ T\to\infty $, we see that $ c_k = 0 $, finally giving us
	\[
	I_k(T) = - \frac{I_k'(1)}{T} + o(1).
	\]
	Our aim is to let $ k\to\infty $, but a priori it is unknown whether the sequence $ \{I_k'(1)\} $ is convergent or not. But since the sequence $ I_k'(1) $ is bounded, we can choose a convergent subsequence. The arguments that follow are valid for any convergent subsequence and therefore we show that the sequence $ \{I_k'(1)\} $ itself is convergent to $ -3/\pi $. With this understanding, without loss of generality we assume that $ I_k'(1) $ is convergent and converges to a constant $ c $. To show that $ c=-3/\pi $, we can proceed as in \cite{Sound Escape of Mass}. Here we present another simpler argument.
	
	We observe that
	\[
	\frac{1}{\|f_k\|^2} \int\limits_{\mathfrak{F}} |f_k(z)|^2 y^{k} dz  = 1,
	\]
	where $ \mathfrak{F} $ denotes the standard fundamental domain. We can approximate the integral over the fundamental domain by integrals over generalized Siegel domains with arbitrarily small widths. With the standard fundamental domain in mind we have, for every $ n \gg 0 $,
	\begin{align*}
		\frac{1}{2}\frac{1}{\|f_k\|^2}\int\limits_{\mathfrak{F}} |f(z)|^2 y^{k} dz &= \sum_{m=0}^{n-1} \int\limits_{-\frac{1}{2}+\frac{m}{2n}}^{-\frac{1}{2} + \frac{m+1}{2n}} \int\limits_{\sqrt{1-(-\frac{1}{2}+\frac{m}{2n})^2}}^{\infty} |f(z)|^2 y^k \frac{dydx}{y^2} + o(1)\\
		&= \frac{1}{2n} \sum_{m=0}^{n-1} I_k\left(\sqrt{1-\left(-\frac{1}{2}+\frac{m}{2n}\right)^2}\right) + o(1)\\
		&= \frac{-I_k'(1)}{2n} \left(\sum_{m=0}^{n-1} \frac{1}{\sqrt{1-\left(-\frac{1}{2}+\frac{m}{2n}\right)^2}}\right) + o(1).
	\end{align*}
	Letting $ n $ go to infinity on the right hand side we get,
	\begin{align*}
		1 +o(1) &= -I_k'(1) \int\limits_{0}^{1} \frac{dt}{\sqrt{1-\left(-\frac{1}{2} + \frac{t}{2}\right)^2}}\\
		&= -I_k'(1) \frac{\pi}{3}.
	\end{align*}
Letting $ k\to\infty $ completes the proof of the theorem.
\end{proof}

Now we are ready to prove Theorem \ref{Theorem "QUE holomorphic"}.

\begin{proof}[Proof of Theorem \ref{Theorem "QUE holomorphic"}]
	Observe that it is sufficient to show \eqref{Equation "QUE holomorphic theorem"} for a dense\footnote{See \S \ref{Section "Description of the method"}} subset of $ C_c^{\infty} (\mathfrak{X}) $. We shall show that \eqref{Equation "QUE holomorphic theorem"} holds for a dense subset of $ L^1(\mathfrak{X}) $.
	
	Until the end of this proof, we shall restrict ourselves to rectangles which are contained inside $ \mathfrak{F} $. Let $ \mathfrak{D} $ denote the set of characteristic functions of rectangles $ \mathcal{R}(z_1,z_2) $. It is clear that $ \mathfrak{D} $ is dense in $ L^1(\mathfrak{X}) $. Furthermore from Theorems \ref{Theorem "QUE Holomorphic Vertical"} and \ref{Theorem "QUE Holomorphic Horizontal"} it is clear that Theorem \ref{Theorem "QUE holomorphic"} holds for every element in $ \mathfrak{D} $.
		
	By abuse of notation let $ \mathfrak{D} $ denote the smallest subspace containing the characteristic functions of rectangles. Let $ \psi\in C_c^\infty(\mathfrak{X}) $, then there exists a sequence of functions $ \rho_i\in \mathfrak{D} $ such that $ \|\rho_i - \psi\|_1\to 0 $. Since $ \psi $ is compactly supported we can in fact choose $ \rho_i $s to also be compactly supported. In fact we choose $ \rho_i $s to be finite linear combinations of characteristic functions of rectangles.
	
	Moreover we can choose $ T $ large enough so that all of the supports of $ \psi $ and $ \rho_i $'s have empty intersection with $ \mathfrak{S}(-\frac{1}{2}, \frac{1}{2}, T) $. Furthermore in the domain, $ \mathfrak{S}(-\frac{1}{2},\frac{1}{2}, \frac{1}{2}) \setminus\mathfrak{S}(-\frac{1}{2},\frac{1}{2}, T) $, the rate of convergence of $ \mu_k $ is uniform. This means that 
	\[
	\lim_{k\to\infty}|\mu_k(\rho_i) - \mu_k(\psi)| \ll \|\rho_i- \psi\|_1,
	\]
	for every $ i $. From Theorems \ref{Theorem "QUE Holomorphic Vertical"} and \ref{Theorem "QUE Holomorphic Horizontal"}, it follows that
	\[
	\lim_{k\to\infty} \mu_k(\rho_i) = \frac{3}{\pi} \int\limits_{\mathfrak{X}} \rho_i dz.
	\]
	Finally we get
	\begin{align*}
		\lim_{k\to\infty} \mu_k(\psi) &= \lim_{i\to\infty} \lim_{k\to\infty} \mu_k(\rho_i)\\
		&= \lim_{i\to\infty} \frac{3}{\pi} \int\limits_{\mathfrak{X}} \rho_i dz\\
		&= \frac{3}{\pi} \int\limits_{\mathfrak{X}} \psi(z) dz,  
	\end{align*}
completing the proof of Theorem \ref{Theorem "QUE holomorphic"}.
\end{proof}

\section{Proofs of Corollaries}\label{Section "Proof of Corollaries"}

\subsection{Lehmer's conjecture}

Suppose $ f_k $ is a Hecke eigenform and $ p $ a prime such that $ a_k(p)=0 $. Fix a $ T > 0 $ and consider the Hecke relation
\[
\frac{1}{p} \sum_{m=0}^{p-1} f_k\left(\frac{m}{p} + \frac{iT}{p}\right) = a_k(p) f_k(p) - p^{k-1} f_k (ipT).
\]
This gives us
\[
\left|\frac{1}{p} \sum_{m=0}^{p-1} f_k\left(\frac{m}{p} + \frac{iT}{p}\right)\right|^2 T^{k-2} = p^{2k-2} |f_k(ipT)|^2 T^{k-2}.
\]

For large enough $ k $ we can rewrite the right hand side above as 
\[
\left|f_k\left(\frac{iT}{p}\right)\right|^2 T^{k-2} + o_T(1) = p^{2k-2}|f_k(ipT)|^2T^{k-2}.
\]
Rewriting this in terms of $ I_k(T) $, we see that
\[
\frac{3p^2}{\pi T^2} + o_T(1) = \frac{3}{\pi T^2},
\]
which is a contradiction for large enough $ k $. This proves Corollary \ref{Corollary "Lehmer"}.

\subsection{Mean Values in Short Intervals}\label{Section "Mean Values in short intervals"}

Continuing the discussion from \S \ref{Section "Description of the method"} and \eqref{Equation "I_k(T) definition 2"} the kernel function in the case of holomorphic cusp forms is the incomplete $ \Gamma $ function. More precisely, we have  $ \Gamma(k-1, 4\pi n T) $ where $ k,T $. It is clear that as $ n\to\infty $, $ \Gamma(k-1,4\pi n T) $ decays exponentially after a stage. But the stage after which the exponential decay occurs is dependent on $ k $ and it increases with $ k $. Therefore the behavior of the incomplete $ \Gamma $ function at the `transition stage', that is $ k \sim 4\pi n T $ becomes important for us.

For every $ k $, we choose a cut off $ \sigma_k $ such that for $ n > \sigma_k $, the behaviour of $ \Gamma(k-1,4\pi n T) $ is better approximated by an exponentially decaying function, and for $ n \leq \sigma_k $, the behaviour of $ \Gamma(k-1, 4\pi nT) $ is better approximated by $ (k-2)! p(k) $ where $ p(x) $ is a polynomial of degree $ k-1 $. The precise relation is given in \eqref{Equation "Incomplete Gamma formula"}. Therefore for $ n > \sigma_k $ the exponential decay of the incomplete $ \Gamma $ function ``cancels out" the polynomial growth of the Fourier coefficients and we can deduce that a certain tail will be negligible as $ k\to\infty $. The cut off $ \sigma_k $ will be of size comparable to $ k/4\pi T $.

We shall follow the notation as in \S \ref{Section "Proof of Theorem 1"} . We start with \eqref{Equation "I_k(T) definition 2"} and divide the summation into two parts. We fix a $ \delta > 0 $ throughout this section.
	\begin{equation}\label{Equation "I_k(T) = M_k(T) + E_k(T)"}
		I_k(T) = M_k(T) + E_k(T),
	\end{equation}
	where
	\begin{align}
		M_k(T) &= \frac{1}{\|f_k\|^2}\frac{1}{(4\pi)^{k-1}} \sum_{n=1}^{(k+k^{1/2+\delta})/(4\pi T)} |\lambda_k(n)|^2 \Gamma(k-1, 4\pi nT),\label{Equation "S_1 definition"}\\
		E_k(T) &= \frac{1}{\|f_k\|^2}\frac{1}{(4\pi)^{k-1}} \sum_{n=(k+k^{1/2+\delta})/(4\pi T)+1}^{\infty} |\lambda_k(n)|^2 \Gamma(k-1, 4\pi nT).\label{Equation "S_2 definition"}
	\end{align}	
	We shall estimate the sums one by one. The term $ M_k(T) $ will be the ``main term" and $ E_k(T) $ will be the ``error term".
		
	We first show that the contribution from $ E_k(T) $ is negligible as $ k\to\infty $. As in the proof of Lemma \ref{Lemma "Gamma function"} we fix an $ \epsilon_1 > 0 $ for the time being. From Deligne's bound on the Fourier coefficients and \eqref{Equation "Bound f_k"}, it follows that,
	\begin{align*}
		\frac{1}{\|f_k\|^2}\frac{1}{(4\pi)^{k-1}} &\sum_{n=\epsilon_1 k/(4\pi T)+1}^{\infty} |\lambda_k(n)|^2 \Gamma(k-1, 4\pi nT) \\
		&\ll \frac{\log(k)}{(k-1)!} \sum_{n >\epsilon_1 k/4\pi T} \tau^2(n) \Gamma(k-1, 4\pi n T)\\
		& \leq \frac{\log(k)}{(k-1)} \sum_{n > \epsilon_1k/4\pi T} n e^{-4\pi nT}\sum_{m=0}^{k-2} \frac{(4\pi nT)^{m}}{m!}\\
		& = \frac{\log(k)}{(k-1)} \sum_{m=0}^{k-2} \frac{(4\pi T)^m}{m!} \sum_{n > \epsilon_1k/4\pi T} n^{m+1} e^{-4\pi nT}.
	\end{align*}
	Since $ m < k $, the terms in the innermost summation is decreasing in $ n $. Therefore the sum can be uniformly bounded above by an integral and we see that the right hand side above is 
	\begin{align*}
		&\ll \frac{\log(k)}{(k-1)} \sum_{m=0}^{k-2} \frac{(4\pi T)^m}{m!} \int\limits_{\epsilon_1k/4\pi T}^{\infty} x^{m+1} e^{-4\pi xT} dx\\
		& = \frac{\log(k)}{(k-1)} \frac{1}{(4\pi T)^2} \sum_{m=0}^{k-2} \frac{1}{m!} \int\limits_{\epsilon_1k}^{\infty} w^{m+1} e^{-w} dw\\
		& = \frac{\log(k)}{(k-1)} \frac{1}{(4\pi T)^2} \sum_{m=0}^{k-2} \frac{1}{m!} \Gamma(m+2,\epsilon_1k).
	\end{align*}
	Treating $ T $ as a constant and using \eqref{Equation "Incomplete Gamma formula"} we can bound the above right hand side by
	\begin{align*}
		&\ll_T \frac{\log(k)}{(k-1)} e^{-\epsilon_1k} \sum_{m=0}^{k-2} (m+1)\sum_{l=0}^{m+1} \frac{(\epsilon_1k)^l}{l!}\\
		&\leq \frac{\log(k)}{(k-1)} e^{-\epsilon_1k} k \sum_{m=0}^{k-2} (m+2) \frac{(\epsilon_1k)^m}{m!}\\
		&\ll \log(k) \left(\frac{e^{-k} k^k}{(k-2)!} \right) \left(\frac{\epsilon_1}{e^{\epsilon_1-1}}\right)^k.
	\end{align*}
	Appealing to Stirling's approximation, the term in the first bracket has at worst polynomial growth in $ k $ and the term in the second bracket has exponential decay in $ k $ as long as $ \epsilon_1 \neq 1 $. Moreover, we can choose $ \epsilon_1 = 1 + k^{-1/2+\delta} $ and the conclusions will hold. Thus $ E_k(T) $ vanishes as $ k\to\infty $.
	
	Now we estimate $ M_k(T) $. Since $ \Gamma(s,x) $ is decreasing in $ x $ for positive values of $ s $, we have
	\begin{align}\nonumber
		M_k(T) = \frac{1}{\|f_k\|^2}\frac{1}{(4\pi)^{k-1}} &\sum_{n=1}^{(k+k^{1/2  + \delta})/(4\pi T)} |\lambda_k(n)|^2 \Gamma(k-1, 4\pi nT)\\
		& < \frac{1}{(4\pi)^{k-1}} \Gamma(k-1, 4\pi T)  \frac{1}{\|f_k\|^2}\sum_{n=1}^{(k+k^{1/2 + \delta})/(4\pi T)} |\lambda_k(n)|^2.
	\end{align}
	
	Similarly for the lower bound we have,
	\begin{multline}
		M_k(T) \geq \frac{1}{\|f_k\|^2}\frac{1}{(4\pi)^{k-1}}\sum_{n=1}^{(k-k^{1/2 + \delta})/(4\pi T)} |\lambda_k(n)|^2 \Gamma(k-1, 4\pi nT) \\
		> \Gamma(k-1, k-k^{1/2+\delta}) \frac{1}{\|f_k\|^2}\frac{1}{(4\pi)^{k-1}} \sum_{n=1}^{(k-k^{1/2 + \delta})/(4\pi T)} |\lambda_k(n)|^2.
	\end{multline}
	
	Now we have
	\begin{multline*}
		\frac{\Gamma(k-1, 4\pi T)}{(4\pi)^{k-1}\|f_k\|^2}\sum_{n=1}^{(k+k^{1/2 - \delta})/(4\pi T)} |\lambda_k(n)|^2 \\
		= \frac{\Gamma(k-1, 4\pi T)}{(4\pi)^{k-1}\|f_k\|^2} \sum_{n=1}^{(k-k^{1/2 - \delta})/(4\pi T)} |\lambda_k(n)|^2  +\\
		\frac{\Gamma(k-1, 4\pi T)}{(4\pi)^{k-1}\|f_k\|^2} \sum_{n=(k-k^{1/2 - \delta})/(4\pi T)}^{(k+k^{1/2 - \delta})/(4\pi T)} |\lambda_k(n)|^2.
	\end{multline*}
	From the Ramanujan bound, the summation in the second term is bounded above by $ k^{1/2 + \epsilon} $ upto some constant independent of $ k $. Also from \eqref{Equation "Bound f_k"}, it follows that the factors outside the summation of the second term decays like $ \log(k)/k^{1/2-\delta} $. Therefore the second term vanishes as $ k\to\infty $. Also from Lemma \ref{Lemma "Gamma function"}, we can replace $ \Gamma(k-1,4\pi T) $ by $ \Gamma(k-1, k-k^{1/2-\delta}) $ as $ k\to\infty $. Combining all this together, we get
	\begin{equation}
		\frac{\Gamma(k-1, 4\pi T)}{(4\pi)^{k-1}\|f_k\|^2}\sum_{n=1}^{(k+k^{1/2 - \delta})/(4\pi T)} |\lambda_k(n)|^2 \sim \frac{\Gamma(k-1, k-k^{1/2-\delta})}{(4\pi)^{k-1}\|f_k\|^2} \sum_{n=1}^{(k-k^{1/2 - \delta})/(4\pi T)} |\lambda_k(n)|^2,
	\end{equation}
	as $ k\to\infty $. It would be convenient to replace $ k-k^{1/2+\delta} $ with $ k $ and it is clear that we can do so.
	
	It follows that 
	\[
	I_k(T) = \frac{\Gamma(k-1, 4\pi T)}{(4\pi)^{k-1}\|f_k\|^2} \sum_{n=1}^{k/4\pi T} |\lambda_k(n)|^2 + o(1).
	\]
	The summation on the right hand side can be written as 
	\[
	\sum_{n=1}^{k/4\pi T} |\lambda_k(n)|^2 = L(1,\sym^2 f_k) \frac{k}{4\pi T} + \mathcal{E}_k,
	\]
	where $ \mathcal{E}_k $ denotes some error term. Substituting this into the earlier equation we have 
	\[
	I_k(T) = \frac{3}{\pi T}\left(1+o(1)\right) + \frac{\mathcal{E}_k}{L(1,\sym^2 f_k) k} + o(1).
	\]
	From Theorem \ref{Theorem "QUE holomorphic"}, it follows that 
	\[
	\mathcal{E}_k = o(kL(1,\sym^2 f_k)),
	\]
	thus completing the proof of Corollary \ref{Corollary "Mean Values"}.

\subsection{A Slight Generalization of Rudnick's Theorem}\label{Subsection "Generalization of Theorem 1"}

We make the crucial observation that following along the same lines as in the proof of Theorem \ref{Theorem "QUE holomorphic"} that we can prove
\begin{multline}\label{Equation "Limit Orthogonality"}
	\lim_{k\to\infty} \frac{1}{\|f_k^{(1)}\|^2} \int\limits_{\mathfrak{F}} f_k^{(1)}(z) \overline{f_k^{(2)}(z)}\psi(z) dz = \lim_{k\to\infty} \frac{1}{\|f_k^{(2)}\|^2} \int\limits_{\mathfrak{F}} f_k^{(1)}(z) \overline{f_k^{(2)}(z)}\psi(z) dz\\
	= \lim_{k\to\infty} \frac{1}{\|f_k^{(1)}\|\|f_k^{(2)}\|} \int\limits_{\mathfrak{F}} f_k^{(1)}(z) \overline{f_k^{(2)}(z)}\psi(z) dz = 0,
\end{multline}
for any two normalized Hecke eigen forms $ f_k^{(1)} $ and $ f_k^{(2)} $ of weight $ k $ and $ \psi\in C_c^\infty(\mathfrak{X}) $. The proof of \eqref{Equation "Limit Orthogonality"} is very similar to the proof of Theorem \ref{Theorem "QUE holomorphic"}. The only point of deviation is in the computation of the constant at the end of proof of Theorem \ref{Theorem "QUE Holomorphic Vertical"} which will turn out to be zero in this case.
Suppose that $ F_k $ is a sequence of admissible forms with notations as in \S \ref{Section "Introduction"}. It follows that
\[
\frac{1}{\|F_k\|^2} \langle \psi y^{k/2}F_k, y^{k/2}F_k\rangle_{\X} = \frac{1}{\|F_k\|^2} \sum_{i=1}^{j_k}\sum_{l=1}^{j_k} \alpha_k^{(i)}\overline{\alpha_k^{(l)}} \langle \psi y^{k/2}f_k^{(i)}, y^{k/2}f_k^{(l)}\rangle_{\X}.
\]
From \eqref{Equation "Limit Orthogonality"} and \eqref{Equation "Bound admissible form"} it follows that the off diagonal terms, that is, the terms corresponding to $ i\neq l $ in the previous summation all vanish as $ k\to\infty $. Therefore we have
\begin{align*}
	\lim_{k\to\infty}\frac{1}{\|F_k\|^2} \langle \psi y^{k/2}F_k, F_k\rangle_{\X} &= \frac{1}{\|F_k\|^2} \sum_{i=1}^{j_k} |\alpha_k^{(i)}|^2 \langle \psi y^{k/2}f_k^{(i)}, y^{k/2}f_k^{(i)}\rangle_{\X}\\
	&=\frac{1}{\|F_k\|^2} \sum_{i=1}^{j_k} |\alpha_k^{(i)}|^2 \|f_k^{(i)}\|^2\left( \frac{1}{\|f_k^{(i)}\|^2}\langle \psi y^{k/2}f_k^{(i)}, y^{k/2}f_k^{(i)}\rangle_{\X}\right).
\end{align*}
Applying Theorem \ref{Theorem "QUE holomorphic"} for the bracketed terms on the right hand side and using \eqref{Equation "Admissible norm"} we see that Corollary \ref{Corollary "Rudnick"} follows.

It was Rudnick who first observed that the equidistribution of mass of holomorphic forms would lead to the equidistribution of the zeroes of the forms themselves \cite{Rudnick - Distribution of zeroes}. Following along the same arguments we are able to prove the following corollaries.

\begin{corollary}\label{Corollary "Admissible zeroes"}
	The zeroes of an admissible sequence of cusp forms equidistribute with respect to $ dz $ as $ k\to\infty $.
\end{corollary}

\begin{corollary}\label{Corollary "Linear combination zeroes"}
	Given a set of non zero complex numbers $ \{\beta_1,\ldots, \beta_j\} $, the solution of the equation 
	\[
	\sum_{i=1}^{j} \beta_i f_k^{(i)}(z) = 0
	\]
	equidistribute with respect to $ dz $ for any choice of normalized eigenforms $ f_k^{(i)} $s.
\end{corollary}

\begin{corollary}\label{Corollary "Powers are not admissible"}
	For a given cusp form $ f $, the sequence $ f^n $ is not an admissible sequence.
\end{corollary}

\subsection{Potential Gradient on $ \mathfrak{X} $}\label{Subsection "Potential Gradient on X"}

In this section we prove Corollary \ref{Corollary "Mass distribution holomorphic"}. But before we do so, we shall give an interpretation in terms of a quantum particle $ p $. Suppose that $ \mathfrak{X} $ is an ambient space in which $ p $ is present in. The state of $ p $ is given by a unit vector in $ L^2(\mathfrak{X}) $. In analogy to the Schr\"odinger equation, we can consider the state of $ p $ as $ \phi_v $ as above. The energy of $ p $ is measured by the size of $ v $. The analogy may be extended in principle to holomorphic forms, in which case the weight $ k $ measures the energy of the particle.

The probability of finding $ p $ in a region $ R\subset \mathfrak{X} $ is given by $ \int_R |\phi_v|^2 dz $. In this interpretation, Corollary \ref{Corollary "Mass distribution holomorphic"} implies that if the energy of $ p $ is bounded the probability of finding $ p $ decays exponentially in $ T $ as we move towards the cusp at infinity. In other words there seems to be a potential gradient ``pulling down" on the particle $ p $. The more energy the particle has, the farther away it can move against the potential gradient.

\begin{proof}[Proof of Corollary \ref{Corollary "Mass distribution holomorphic"}]
	We shall directly prove the second inequality in the statement of Corollary \ref{Corollary "Mass distribution holomorphic"}. Fix $ k $ and $ f_k $ as before and suppose that $ T\geq 4k\log(k) $. Then it follows that for every $ x > 4\pi T $, we have $ x^{2k-2} e^{-x} \leq e^{-x/2} $. Therefore we have $ \Gamma(k-1, 4\pi n T) \leq 2 e^{-2\pi n T}$. Plugging this into \eqref{Equation "I_k(T) definition 2"} we get
	\begin{align*}
		\mu_k\left(\mathfrak{S}\left(-\frac{1}{2},\frac{1}{2},T\right)\right) &= \frac{1}{\|f_k\|^2}\frac{1}{(4\pi)^{k-1}} \sum_{n=1}^{\infty} |\lambda_k(n)|^2 \Gamma(k-1, 4\pi nT)\\
		&\leq \frac{2}{\|f_k\|^2}\frac{1}{(4\pi)^{k-1}} \sum_{n=1}^{\infty} |\lambda_k(n)|^2 e^{-2\pi n T}\\
		&\leq \sum_{n=1}^{\infty}n e^{-2\pi n T}\\
		&\leq \frac{e^{-2\pi T}}{2\pi T}.
	\end{align*}
	This proves the corollary.
\end{proof}

\section{Concluding Remarks}\label{Section "Concluding remarks"}

\begin{remark}
	The proof of Theorem \ref{Theorem "QUE holomorphic"} can easily be extended to forms of higher level. Suppose that $ f_k $'s are a sequence of cuspidal Hecke eigenforms of level $ N > 1 $ and of weight $ k $. The only point of deviation would be in the choice of $ (p,q) $ as in the proof of Theorem \ref{Theorem "QUE Holomorphic Vertical"}. We need some criterion to distinguish primes which can be written of the form $ c^2N^2 + d^2 $ for two integers $ c,d $. Using standard arguments from class field theory and in particular an effective version of the Chebatrov density theorem, we can show that there is a positive proportion of primes which can be represented by the above quadratic form. Therefore arguing along the same lines as in \cite{Micholson}, it also follows that the quotient of those primes are dense in $ \R $. Now the proof of Theorem \ref{Theorem "QUE Holomorphic Vertical"} can be suitably modified.
\end{remark}

\begin{remark}\label{Remark "Maass forms"}
	Extending these methods for the case of Maass forms would probably require a careful analysis of the $ K $-Bessel functions (similar to our considerations of the incomplete $ \Gamma $ function). Nevertheless the second part of the proof of Theorem \ref{Theorem "QUE holomorphic"}, in particular deducing the vertical equidistribution from a suitble form of horizontal equidistribution and the deduction of the QUE conjecture thereon, is applicable for the case of Maass forms as well.
\end{remark}

\begin{remark}\label{Remark "Ramanujan"}
	It is clear from the proof that we do not require the full strength of the Ramanujan bound. It is enough that we have a smaller power saving. For example, a bound of the form $ a_k(n) \ll n^{\frac{k}{2} - \frac{3}{8} - \frac{1}{8\pi}} $ should suffice. More importantly, the bound should be uniform in $ k $.
\end{remark}

\begin{remark}\label{Remark "Varying level"}
	It would be interesting to find an analogue of the present arguments to prove the ``level equidistribution" result of Nelson \cite{Nelson - QUE varying level}.
\end{remark}

\begin{remark}\label{Remark "Hoffstein Lockhart"}
	If $ f_k $ is a Hecke eigenform, the trivial lower bound on $ L(1,\sym^2 f_k) $ is of the size $ k^{-\epsilon} $ for an arbitrarily small positive quantity $ \epsilon $. We remark here that this would suffice for our purpose of bounding the error term in the horizontal equidistribution.
\end{remark}

\begin{remark}
	The author hopes that the method of proof will be extended to other surfaces which do not have an arithmetic nature, in particular to those surfaces which do not possess an action of a Hecke ring, since we have very minimally used the property that the forms considered are Hecke eigenforms.
\end{remark}

\begin{remark}
	As observed above, suconvexity bounds for the relevant $ L $ functions would prove the QUE conjecture. More precisely, improvements towards the Lindel\"off hypothesis would lead to faster rates of convergence.
\end{remark}

\begin{remark}
	It seems that a similar argument can be used to prove the QUE conjecture for half integral weight modular forms and is currently under consideration by the author.
\end{remark}

\begin{remark}
	It would be quite interesting to give an effective version of Corollary \ref{Corollary "Lehmer"}.
\end{remark}

\end{document}